\documentclass[preprint,review,10pt]{amsart}

\newtheorem{theorem}{Theorem}[section]
\newtheorem{lemma}[theorem]{Lemma}

\theoremstyle{definition}
\newtheorem{definition}[theorem]{Definition}

\newtheorem{proposition}[theorem]{Proposition}
\newtheorem{corollary}[theorem]{Corollary}

\theoremstyle{remark}
\newtheorem{remark}[theorem]{Remark}

\numberwithin{equation}{section}

\begin{document}

\title{The Ricci-Bourguignon flow on Heisenberg and quaternion Lie  groups  }

\author{Shahroud Azami}
\address{Department of Mathematics, Faculty of Sciences, Imam Khomeini International University, Qazvin, Iran.
}

\email{azami@sci.ikiu.ac.ir,\,\, shahrood78@yahoo.com }

%\author{}
%\address{}
%\email{}

%    General info

%\date{September  , 2013.}

%------------------------------------------------------------------------------
%%%%%%%%%%%%%%%%%%%%%%%%%%%%%%%%%%%%%%%%%%%%%%%%%%%%%%%%%%%%%%%%%%%%%%%%%%%%%%%%%%%%%%%%%%%%%%%%%%%%%%

%%%%%%%%%%%%%%%%%%%%%%%%%%%%%%%%%%%%%%%%%%%%%%%%%%%%%%%%%%%%%%%%%%%%%%%%%%%%%%%%%%%%%%%%%%%%%%%%%%%%%%%%%%%%%%%%%%%%%%%%%%%%%%%%%%%%%%%%%
\begin{abstract}
In this paper, we study the Ricci-Bourguignon flow on higher dimensional classical Heisenberg nilpotent Lie groups and construct a solution of this flow on Heisenberg and quaternion nilpotent Lie groups. In the end, we investigate the deformation of spectrum and length spectrum on compact nilmanifolds obtained of  Heisenberg and quaternion nilpotent Lie groups.
\end{abstract}
%%%%%%%%%%%%%%%%%%%%%%%%%%%%%%%%%%%%%%%%%%%%%%%%%%%%%%%%%%%%%%%%%%%%%%%%%%%%%%%%%
\maketitle
\section{Introduction}
Geometric flow is an evolution of a geometric structure under a differential equation associated with some curvature and it is an important topic in many branches of mathematics and physcis. A geometric flow is  related to dynamical systems in the
 infinite-dimensional space of all metrics on a given manifold. \\

%%%%%%%%%%%%%%%%%%%%%%%%%%%%%%%%%%%%%%%%%%%%%%%%%%%%%%
Let $M$ be an $n$-dimensional  manifold with a Riemannian metric $g_{0}$, the  family
${g(t)}$ of Riemannian metrics on $M$
 is called a   Ricci-Bourguignon flow when it satisfies the equations\\
\begin{equation}\label{rb}
\frac{d}{dt}g(t)=-2Ric(g(t))+2\rho R(g(t))g(t)=-2(Ric-\rho Rg) ,\,\, g(0)=g_{0}\\
\end{equation}
where $Ric$  is the Ricci tensor of $g(t)$, $R$ is the scalar curvature and $\rho$ is a real constant.  In fact the Ricci-Bourguignon  flow is a system of partial differential equations  which was introduced   by Bourguignon for the first time in
 1981 (see \cite{JPB}).
For closed manifolds, short time
existence and uniqueness for
solution to the Ricci-Bourguignon flow on $[0,T)$ have been shown  by Catino and et'al in  ~\cite{GC} for $\rho<\frac{1}{2(n-1)}$.
When  $\rho=0$,
the Ricci-Bourguignon flow is the Ricci flow.\\

%%%%%%%%%%%%%%%%%%%%%%%%%%%%%%%%%
%%%%%%%%%%%%%%%%%%%%%%%%%%%%%%%%%%
A Riemannian metric $g$ on $N $ is left invariant
if the left translations $L_{p}$'s are isometries for all $p\in  N$.  We will use $< , >$ to denote both the inner product on
$\mathcal{N} = T_{e}N$ and the corresponding left invariant metric on $N$. Let $\mathcal{Z}$ be the center of $\mathcal{N}$, we
denote the orthogonal complement of $\mathcal{Z}$ in $\mathcal{N}$
by $\mathcal{V}$ and write $\mathcal{N} =
\mathcal{V}\oplus\mathcal{Z}$. Define a linear
transformation $ j:\mathcal{Z}\rightarrow SO(\mathcal{V})$ by $j(Z)X=(adX)^{*}Z$ for $Z\in\mathcal{Z}$ and $X\in\mathcal{V}$. Equivalently, for each $Z\in\mathcal{Z}$, $ j(Z):\mathcal{V}\rightarrow\mathcal{V}$ is the
skew-symmetric linear transformation defined by
\begin{equation}
<(adX)^{*}Z,Y>=<Z,(adX)Y>,
\end{equation}
 for all $X,Y\in\mathcal{V}$. Here  $adX(Y) = [X,Y]$ for all
$X,Y\in\mathcal{N}$, and $(adX)^{*}$ denotes the (metric) adjoint of $adX$.
 A 2-step nilpotent   Lie algebra $\mathcal{N}$ is called of  Heisenberg  type if  ${j(Z)}^{2}=-\left|Z\right|^{2}Id $ for all $Z\in\mathcal{Z}$, for instance,  the classical Heisenberg Lie group $H_{n}$  and quaternion Lie group $Q_{n}$  with  special metrics are of   Heisenberg  type (see \cite{E1, E2,   He}).\\

Lauret in \cite{JLL}, studied the Ricci soliton on homegenous nilmanifolds and then Payne in \cite{PPY, PPY1} investigated the Ricci flow and the soliton metrics on  nilmanifollds and nilpotent lie groups. Also, Williams in \cite{W} funded the explicit solution for  the Ricci flow on some nilpotent Lie groups for instance, the classical Heisenberg Lie group $H_{n}$ of dimension $(2n+1) $. Author and Razavi in \cite{SA}, studied  the eigenvalue variations of Heisenberg and quaternion Lie groups under the Ricci flow and  investigated the deformation of some characteristics  of  compact nilmanifolds $\Gamma\setminus N$ under the Ricci flow, where $N$ is a simply connected $2$-step nilpotent Lie group with a left invariant metric and $\Gamma$ is a discrete cocompact subgroup of $N$, in particular  Heisenberg  and quaternion Lie groups.\\

Motivated by the above works, in this paper, we will investigate the  Ricci-Bourguignon flow on higher  dimensional classical Heigenberg and quaternion nilpotent Lie groups and find the deformation of spectrum and  length spectrum of compact nilmanifold obtained of the Heigenberg and quaternion nilpotent Lie groups.
%%%%%%%%%%%%%%%%%%%%%%%%%%%%%%%

\section{Preliminaries}
\subsection{Curvature of Lie groups}
We recall some properties about the geometry of Lie groups with
left-invariant metrics, and derive the formula for the Ricci tensor
(see \cite{B, He, M}). Suppose  that $< . , .>$ is a
left-invariant metric on a Lie group $N$, which is equivalent to an
inner product on the Lie algebra $\mathcal{N}$. Let $\nabla$ denote
the Levi-Civita connection for the metric, and let $X,Y,Z\in
\mathcal{N}$.
We shall need the following useful theorems and propositions of
(\cite{B}) about the Ricci tensor of a
Lie group.
\begin{proposition}\label{p1}
Let $<. , .>$ be a left-invariant metric on a Lie group $N$ and $\nabla$ the connection for this metric, for $X,Y,Z,W\in \mathcal{N}$, we have\\
{\bf a})$
\nabla _X Y = \frac{1}{2}\left\{ {\left( {adX} \right)Y - \left( {adX} \right)^* Y - \left( {adY} \right)^* X} \right\}
$,\\
{\bf b)}
$
 < R(X,Y)Z,W >  = < \nabla _X Z,\nabla _Y W >  -  < \nabla _Y Z,\nabla _X W
 >-  < \nabla _{\left[ {X,Y} \right]} Z,W >.
$

\end{proposition}
Besides, the maps
\begin{equation}
\left( {X,Y} \right) \mapsto \left( {adX} \right)Y,\begin{array}{*{20}c}
   {} & {\left( {X,Y} \right) \mapsto \left( {adX} \right)^* Y},
\end{array}
\end{equation}
  from $\mathcal{N}\times\mathcal{N}$ to $\mathcal{N}$ are bilinear maps. We define \\
\begin{equation}
\begin{array}{l}
 U:\mathcal{ N}  \times \mathcal{ N}  \to \mathcal{ N}  \\
 U(X,Y) =  - \frac{1}{2}\left\{ {\left( {adX} \right)^* Y + \left( {adY} \right)^* X} \right\}, 
 \end{array}
\end{equation}
which is bilinear and symmetric.
%%%%%%%%%%%%%%%%%%%%%%%%%%%%%%
\begin{proposition}\label{p2}
The Riemannian curvature tensor on $N$ is given by
\begin{eqnarray}
 4 < R(X,Y)Z,W >  &=& 2 < \left[ {X,Y} \right],\left[ {Z,W} \right] >
 +  < \left[ {X,Z} \right],\left[ {Y,W} \right] >\nonumber\\
 & &-  < \left[ {X,W} \right],\left[ {Y,Z} \right] >  -  < [[X,Y],Z],W > \\
 && +  < [[X,Y],W],Z >  -  < [[Z,W],X],Y >\nonumber \\& &
  +  < [[Z,W],Y],X >  + 4 < U(X,Z),U(Y,W) >\nonumber\\& &
    - 4 < U(X,W),U(Y,Z) >.\nonumber
\end{eqnarray}
In particular,
\begin{eqnarray}
  < R(X,Y)Z,X >  &=& \frac{1}{4}\left\| {\left( {adX} \right)^* Y + \left( {adY} \right)^* X} \right\|^2
  -  < \left( {adX} \right)^* X,\left( {adY} \right)^* Y > \nonumber \\
  & &- \frac{3}{4}\left\| {\left[ {X,Y} \right]} \right\|^2  - \frac{1}{2} < [[X,Y],Y],X
  >\nonumber\\
  & &- \frac{1}{2} < [[Y,X],X],Y >.
\end{eqnarray}
\end{proposition}
%%%%%%%%%%%%%%%%%%%%%%%%%%%
Now, suppose that $\{e_{i}\}$ is a basis for the Lie algebra $\mathcal{N}$, then we write:
\begin{equation}
\left( {ade_i } \right)e_j  = C_{ij}^k e_{k} ,\begin{array}{*{20}c}
   {} & {\left( {ade_i } \right)^* e_j  = a_{ij}^k e_{k} },  \\
\end{array}\begin{array}{*{20}c}
   {} & { < e_i ,e_j  >  = g_{ij} }.  
\end{array}
\end{equation}
%%%%%%%%%%%%%%%%%%%%%%%
It yields;
\begin{corollary}\label{c1}
{\bf a)}
$a_{ij}^k  = C_{il}^m g_{jm}g^{kl}$,\\
{\bf b)} If\, $\nabla _{e_i } e_j  = \gamma _{ij}^k e_k\,$ then
\begin{equation}
\gamma _{ij}^k  = \frac{1}{2}g^{kl} \left( {C_{ij}^m g_{lm}  - C_{il}^m g_{jm}  - C_{jl}^m g_{im} } \right).
\end{equation}
{\bf c)} The components of the Riemann curvature tensor satisfy
\begin{eqnarray}
 4R_{ijkl}  &=& 2C_{ij}^p C_{kl}^q g_{pq}  + C_{ik}^p C_{jl}^q g_{pq}  - C_{il}^p C_{jk}^q g_{pq}
  - C_{ij}^p C_{pk}^q g_{ql}  + C_{ij}^p C_{pl}^q g_{pk}\nonumber\\
    &&- C_{kl}^p C_{pi}^q g_{qj}  + C_{kl}^p C_{pj}^q g_{qi}
  + \left( {a_{ik}^p  + a_{ki}^p } \right)\left( {a_{jl}^q  + a_{lj}^q } \right)g_{pq}\\ &&
   - \left( {a_{il}^p  + a_{li}^p }
   \right)\left( {a_{jk}^q  + a_{kj}^q } \right)g_{pq}. \nonumber
 \end{eqnarray}
{\bf d)} The components of the Ricci curvature tensor satisfy
\begin{eqnarray}\label{m21}
 4R_{ij}  &=& \left\{ {2C_{ki}^p C_{jm}^q g_{pq}  + C_{kj}^p C_{im}^q g_{pq}  - C_{km}^p C_{ij}^q g_{pq}
   - C_{ki}^p C_{pj}^q g_{qm}   } \right.\nonumber\\
 && + C_{ki}^p C_{pm}^q g_{qj}  - C_{jm}^p C_{pk}^q g_{qi}  + C_{jm}^p C_{pj}^q g_{qk}\\&&
 + \left( {a_{jk}^p  + a_{kj}^p } \right)\left( {a_{im}^q  + a_{mi}^q } \right)g_{pq}
  - \left( {a_{km}^p  + a_{mk}^p } \right)\left( {a_{ij}^q  + a_{ji}^q } \right)\left.
  {g_{pq} } \right\}g^{km}. \nonumber
\end{eqnarray}

\end{corollary}
%%%%%%%%%%%%%%%%%%%%%%%%%
\subsection{Heisenberg Lie group}
We now recall the construction and properties of the higher-dimensional, classical Heisenberg Lie  group.
Let $H_{n}$  be a $(2n+1)$-dimensional Heisenberg Lie group.
Let
\begin{eqnarray*}
  x &=& (x^{1},...,x^{n}) \\
  y &=& (x^{n+1},...,x^{2n}).
\end{eqnarray*}
If $q=(x,y,z)\in H_{n}$ and  $q=(x',y',z')\in H_{n}$ then the group  multiplication is
\begin{equation*}
  (x,y,z)\circ(x',y',z')=(x+x',y+y',z+z'+x.y')
\end{equation*}
where $x.y'$ is the usual inner product of vectors $x\in \mathbb{R}^{n}$ and $y'\in \mathbb{R}^{n}$. With  respect to this multiplication, we have the following  frame of left invariant vector fields,
\begin{equation*}
  e_{i} =\partial_{i},\,\,\,e_{n+i} = \partial_{n+i}+x^{i}\partial_{2n+1},\,\,\,e_{2n+1} = \partial_{2n+1},\,\,\text{for all $1\leq i\leq n$},
\end{equation*}
and the only nontrivial Lie bracket relation is
\begin{equation*}
  [e_{i},e_{n+i}]=e_{2n+1},\,\,\,\,\text{for all $1\leq i\leq n$}.
\end{equation*}
The dual coframe is
\begin{equation*}
\theta^{i} =dx^{i},\,\,\,\theta^{n+i} =dx^{n+i},\,\,\,\theta^{2n+1} =dx^{2n+1},\,\,\,\,\text{for all $1\leq i\leq n$}.
\end{equation*}
Set $\mathcal{V}=\text{span}\{e_{i}, e_{n+i}|\,\, 1\leq i\leq n \}$ and $\mathcal{Z}=\text{span}\{e_{2n+1} \}$. With the above multiplication $\mathcal{V}\cup \mathcal{Z}$ is an orthonormal basis for $\mathcal{H}_{n}$, then  $\mathcal{H}_{n}=\mathcal{V}\oplus \mathcal{Z}$ and  the Heisenberg Lie group is of Heisenberg type.

%%%%%%%%%%%%%%%%%%%%%
\subsection{The Ricci-Bourguignon  flow  on  the Heisenberg Lie group}
In this section, we study solutions to The Ricci-Bourguignon  flow (\ref{rb}) starting at some initial metric $g_{0}$ on Heisenberg Lie group. Any one-parameter family of left invariant metrics $g(t)$ on ${H}_{n}$ which is a solution of the Ricci-Bourguignon  flow, can be written as 
$$g(t)=g_{IJ}(t)\theta^{I}\otimes\theta^{J}.$$
 In \cite{W}, Williams, using Propositions  \ref{p1}, \ref {p2}  and Corollary \ref{c1}, showed that  the Ricci tensor of $H_{n}$ as follows:
\begin{equation*}
  \begin{cases}
R_{ij}(t)=-\frac{1}{2}g^{i+n,j+n}(t)g_{NN}(t)+\frac{1}{2}g_{iN}(t)g_{jN}(t)\sum ,&\text{if $1 \leq i , j \leq n$;}\\
R_{i,j+n}(t)=\frac{1}{2}g^{i+n,j}(t)g_{NN}(t)+\frac{1}{2}g_{iN}(t)g_{j+n,N}(t)\sum ,&\text{if $1 \leq i , j \leq n$;}\\
R_{iN}(t)=\frac{1}{2}g_{iN}(t)g_{NN}(t)\sum ,&\text{if $1 \leq i \leq n$;}\\
R_{i+n,j+n}(t)=-\frac{1}{2}g^{ij}(t)g_{NN}(t)+\frac{1}{2}g_{i+n,N}(t)g_{j+n,N}(t)\sum ,&\text{if $1 \leq i , j \leq n$;}\\
R_{i+N,N}(t)=\frac{1}{2}g_{i+n,N}(t)g_{NN}(t)\sum ,&\text{if $1 \leq i \leq n$;}\\
R_{NN}(t)=\frac{1}{2}g^{2}_{NN}(t)\sum.
\end{cases}
\end{equation*}
which
\begin{equation*}
  \sum=\sum_{k,m=1}^{n}g^{km}(t)g^{k+n,m+n}(t)-\sum_{k=1}^{n}\sum_{m=n+1}^{2n}g^{km}(t)g^{k+n,m-n}(t).
\end{equation*}
 and $N=2n+1$. We assume that the Riemannian metric initial is diagonal. From now on, we only use single subscripts for the metric components: $ g_{1}(t),...,g_{N}(t)$.
  This implies that the Ricci tensor stays diagonal under the Ricci-Bourguignon flow, and  the Ricci tensor as follows:
\begin{equation*}
  \begin{cases}
R_{ij}(t)=\begin{cases}-\frac{1}{2}g^{i+n}(t)g_{N}(t) &\text{if $ i=j$}\\ 0 &\text{if $ i\neq j$}\end{cases},&\text{if $1 \leq i , j \leq n$;}\\
R_{i,j+n}(t)=0 ,&\text{if $1 \leq i , j \leq n$;}\\
R_{iN}(t)=0 ,&\text{if $1 \leq i \leq n$;}\\
R_{i+n,j+n}(t)=\begin{cases}-\frac{1}{2}g^{i}(t)g_{N}(t)&\text{if $ i=j$}\\ 0 &\text{if $ i\neq j$}\end{cases} ,&\text{if $1 \leq i , j \leq n$;}\\
R_{i+N,N}(t)=0 ,&\text{if $1 \leq i \leq n$;}\\
R_{NN}(t)=\frac{1}{2}g^{2}_{N}(t)\sum.
\end{cases}
\end{equation*}
By direct computation we obtaine the scaler curvature as follows:
\begin{equation*}
  R(t)=-\frac{1}{2}g_{N}\sum.
\end{equation*}
Then the Ricci-Bourguignon flow equation on $H_{n}$ with a diagnoal left-invariant metric $g_{0}$ has the following form
\begin{equation}\label{R5}
  \begin{cases}
 \frac{d}{dt}g_{i}(t)= \frac{g_{N}(t)}{g_{i+n}(t)}-\rho g_{i}(t)g_{N}(t)\sum, &\text{if $1 \leq i  \leq n$;}\\
  \frac{d}{dt}g_{i+n}(t)= \frac{g_{N}(t)}{g_{i}(t)}-\rho g_{i+n}(t)g_{N}(t)\sum, &\text{if $1 \leq i  \leq n$;}\\
   \frac{d}{dt}g_{N}(t)= -(1+\rho)g^{2}_{N}(t)\sum.
\end{cases}
\end{equation}
Let $g_{1}, g_{2}, ..., g_{2n}, g_{N}$ be a solution of the Ricci-Bourguignon flow.  As diagonal components of a metric, they are positive function of $t$.
\begin{theorem}
Consider the Heisenberg Lie group $H_{n}$ with a diagonal  left-invariant metric $g_{0}$. Let $g(t)$ be a solution to the Ricci-Bourguignon flow with initial metric $g_{0}$, then
\begin{description}
  \item[a] $\frac{d}{dt}\frac{g_{i}(t)}{g_{i+n}(t)}=0,\,\,\text{if $1 \leq i  \leq n$;}$
  \item[b]  $\frac{d}{dt}\big(g_{1}(t)...g_{n}(t)g_{N}^{\frac{1-n\rho}{1+\rho}}(t)\big)=\frac{d}{dt}\big(g_{1+n}(t)...g_{2n}(t)g_{N}^{\frac{1-n\rho}{1+\rho}}(t)\big)=0$, 
  \item[c] If $\rho<0$  and $G_{N}(t)=\int_{0}^{t}g_{N}(t)dt$  then $\mathop{\lim}\limits_{t\to+\infty}G_{N}(t)=+\infty$. 
\item[d] Moreover, if  $g_{i}(0)g_{n+i}(0)=g_{1}(0)g_{1+n}(0)$, for $1\leq i\leq n$ then a sloution $g(t)$ has the following form
\begin{equation}\label{ss}
  \begin{cases}
g_{j}(t)=g_{j}(0)\Big( 1+bt  \Big)^{\frac{1-n\rho}{n+2-n\rho}},&\text{if\,\,\,\,$1\leq j\leq 2n$} \\
g_{N}(t)=g_{N}(0)\Big( 1+bt  \Big)^{\frac{n+n\rho}{n\rho-n-2}}
\end{cases}
\end{equation}
where $b=(n+2-n\rho)\frac{g_{N}(0)}{g_{1}(0)g_{1+n}(0)}$.
\end{description}
\end{theorem}
\begin{proof}
\begin{description}
  \item[a] 
 Using (\ref{R5}) and direct computating we have
\begin{equation*}
\frac{d}{dt}\frac{g_{i}(t)}{g_{i+n}(t)}=0.
\end{equation*}
\item[b] By diffrential  respect to variabla time  $t$ and using (\ref{R5})  we obtain
\begin{eqnarray*}
\frac{d}{dt}\big( g_{1}(t)...g_{n}(t)(g_{N}(t))^{\frac{1-n\rho}{1+\rho}}  \big)&=& \big( \sum_{k=1}^{n}\frac{1}{g_{k}(t)}\frac{dg_{k}(t)}{dt}+\frac{1-n\rho}{1+\rho}\frac{1}{g_{N}(t)}\frac{dg_{N}(t)}{dt} \big)g_{1}(t)...g_{n}(t)g_{N}^{\frac{1-n\rho}{1+\rho}}\\
&=&\big( \sum_{k=1}^{n}(\frac{g_{N}(t)}{g_{k}(t)g_{n+k}(t)}-\rho g_{N}(t)\sum)\Big)g_{1}(t)...g_{n}(t)(g_{N}(t))^{\frac{1-n\rho}{1+\rho}}\\
&&-\Big((1-n\rho)g_{N}(t)\sum\Big) g_{1}(t)...g_{n}(t)(g_{N}(t))^{\frac{1-n\rho}{1+\rho}}=0,
\end{eqnarray*}
the part $\bf{a}$ of Theorem implies that $\frac{g_{i}(t)}{g_{i+n}(t)}$ is constant for $1\leq i\leq n$, so we can set 
\begin{equation*}
  A_{i}=\frac{g_{i}(t)}{g_{i+n}(t)}=\frac{g_{i}(0)}{g_{i+n}(0)},
\end{equation*}
therefore $ g_{i+n}(t)=\frac{g_{i}(t)}{A_{i}}$.  Hence $\frac{d}{dt}\big( g_{1}(t)...g_{n}(t)(g_{N}(t))^{\frac{1-n\rho}{1+\rho}}  \big)=0$ results that $\frac{d}{dt}\big(g_{1+n}(t)...g_{2n}(t)g_{N}^{\frac{1-n\rho}{1+\rho}}(t)\big)=0$. This completes the proof   part $\bf{b}$ of the Theorem.
\item[c] For $\rho<0$ the equations (\ref{R5}) implies that $g_{j},\,\,\, 1\leq j \leq 2n$ is an increasing function, so $\sum$ 
is positive and decreasing. Since $g_{N}(t)$ is positive  then last equation in  (\ref{R5}) yields
  \begin{equation*}
\frac{d}{dt}g_{N}(t)= -(1+\rho)g^{2}_{N}(t)\sum\geq -g^{2}_{N}(t)\sum\geq -\sum(0)g^{2}_{N}(t),
\end{equation*}
which by direct computation results that  
\begin{equation*}
 g_{N}(t)\geq \frac{1}{\sum(0)t+g^{-1}_{N}(0)},
\end{equation*}
by this it holds that 
\begin{equation*}
\mathop{\lim}\limits_{t\to+\infty}G_{N}(t)=\mathop{\lim}\limits_{t\to+\infty}\int_{0}^{t}g_{N}(r)dr\geq \mathop{\lim}\limits_{t\to+\infty}\int_{0}^{t}\frac{1}{\sum(0) r+g^{-1}_{N}(0)}dr= +\infty,
\end{equation*}
this proves part $\bf{c}$ of the Theoorem.
\item[d] $g_{j}(t)$ and $g_{N}(t)$ for $1\leq j\leq 2n$  given in (\ref{ss}) satisfy in  the Ricci-Bourguignon flow (\ref{R5}).
\end{description}
\end{proof}
%%%%%%%%%%%%%%%%%%%%%%%%%%
 Consider
\[
\mathcal{Z} = \text{span }\left\{ {e_{2n + 1} }\right\},\begin{array}{*{20}c}
   & {}  \\
\end{array}\mathcal{V} =  \text{span }\left\{ {e_1 ,e_2 ,...,e_{2n} } \right\},\begin{array}{*{20}c}
   {} & {\mathcal{H}_n  = \mathcal{V} \oplus \mathcal{Z}},  \\
\end{array}
\]

 where $\mathcal{Z}$ is the  center of $\mathcal{H}_{n}$
 and $\mathcal{V}$ is the orthogonal complement of $\mathcal{Z}$ in $\mathcal{H}_{n}$. If $Z=e_{2n+1}$ then
 \[
Z =  e_{2n + 1} ,\begin{array}{*{20}c}
   {} & {j(Z)e_i  = e_{n + i} } , \\
\end{array}\begin{array}{*{20}c}
   {} & {j(Z)e_{n + i}  = -e_{i} },  \\
\end{array}
\]
hence
\begin{equation}
j(Z) = \left[ {\begin{array}{*{20}c}
   o & { - I_n }  \\
   {I_n } & o  \\
\end{array}} \right],\,\,\, \left( {j(aZ)} \right)^2  = \left[ {\begin{array}{*{20}c}
   { - a^2 I_n } & o  \\
   o & { - a^2 I_n }  \\
\end{array}} \right],
\end{equation}
where $I_{n}$ is an $n\times n$ identity matrix and it yields to 
\begin{equation}
\left( {j(aZ)} \right)^2  =  - \left| {aZ} \right|^2 Id,
\end{equation}
therefore $\mathcal{H}_{n}$ with this structure is of Heisenberg
type.\\
%%%%%%%%%%%%%%%%%%%%%%%%%%%%%%%%
\begin{proposition}\label{p4}
Heisenberg type of Lie group ${H}_{n}$ is not preserved along  the  Ricci-Bourguignon flow with solution (\ref{ss}) with additional condition $g_{2n+1}(0)=g_{i}(0)g_{n+i}(0)$, for $1\leq i\leq n$.
\end{proposition}
\begin{proof} 
In $(\mathcal{H}_{n}, g_{t}=< , >_{t})$, if  we assume that $j(Z)e_i = \sum\limits_{k = 1}^{2n + 1} {a_k (t)e_k }$,  $1 \le i \le n$ then 
$$ < Z,[e_i ,e_j ] > _t  =< j(Z)e_i ,e_j  > _t  = \sum\limits_{k = 1}^{2n + 1} {a_{k}(t)  < e_k ,e_j  > _{t} }=a_j  (t)< e_j ,e_j  > _{t},$$
 therefore
\begin{equation*}
 j(Z)e_i  = \sum\limits_{j = 1}^{2n + 1} {\frac{{ < Z,[e_i ,e_j ] > _t }}{{ < e_j ,e_j  > _t }}} e_j  = \frac{|Z|_{t}^{2}}{g_{n + i} (t)}e_{n+i},
 \end{equation*}
in the same way it  can be shown that
\[
j(Z)e_{n + i}  =- \frac{|Z|_{t}^{2}}{g_{i}(t)} e_{i},\,\,\,\, 1\leq i\leq n,
\]
hence 
\[
j(Z) = A\left[ {\begin{array}{*{20}c}
   o & {-B_{1} }  \\
   {B_{2} } & o  \\
\end{array}} \right],
\]
where $B_{1}=diag(\frac{1}{g_{1}(t)},...,\frac{1}{g_{n}(t)})$, $B_{2}=diag(\frac{1}{g_{n+1}(t)},...,\frac{1}{g_{2n}(t)})$, $A=|Z|_{t}^{2}$. Now for any real constant $a$ we obtain 
\[
j(aZ) =aA \left[ {\begin{array}{*{20}c}
   o & {-B_{1} }  \\
   {B_{2} } & o  \\
\end{array}} \right],\,\,\, \left( {j(aZ)} \right)^2  = -a^{2}A^{2}\left[ {\begin{array}{*{20}c}
  D & o  \\
   o & D  \\
\end{array}} \right],
\]
where $D=diag(\frac{1}{g_{1}(t)g_{n+1}(t)},...,\frac{1}{g_{n}(t)g_{2n}(t)})$. But for $1\leq i\leq n$ we have $g_{i}(0)g_{n+i}(0)=g_{2n+1}(0)$, then  (\ref{ss}) results that 
\[
\left( {j(Z)} \right)^2  = -\frac{1}{(n+2-n\rho)t + 1} \left| Z \right|_{t}^2 I_{2n}.
\]
So, Heisenberg type of $H_{n}$ is not preserved under the evolution of the Ricci-Bourguignon flow.
\end{proof}
%%%%%%%%%%%%%%%%%%%%%%%%%%%%%%
 \begin{definition}\label{ddh}
{\bf a)} Let $\mu(Z)$ denote the number of distinct eigenvalues of $j(Z)^{2}$ and $-\theta_{1}(Z)^{2},-\theta_{2}(Z)^{2},...,-\theta_{\mu}(Z)^{2}$ denote the $\mu$ distinct eigenvalues of $j(Z)^{2}$,with the assumption that 
 $0\leq\theta_{1}(Z)<\theta_{2}(Z)<...<\theta_{\mu}(Z)$.\\
{\bf b)} A two-step
nilpotent metric Lie algebra $(\mathcal{N}, <, >)$ is
Heisenberg-like if
 $[j(Z)X_{m},X_{m}]\in span_{\mathbb{R}}{Z}$ for all $Z\in \mathcal{Z}$ and all $X_{m}\in W_{m}(Z), m=1,...,\mu(Z)$,
  where $W_{m}$
 denotes the invariant subspace of $j(Z)$ corresponding to $\theta_{m}(Z),\,\,m=1,...,\mu(Z)$.
 \end{definition}
 If $\mathcal{N}$ is of Heisenberg type, then for all $ Z\in \mathcal{Z}$ and $X\in \mathcal{V}$,
 \begin{equation}
 [X,j(Z)X]=\mid X\mid^{2}Z.
  \end{equation}
   If $\mathcal{N}$ is of Heisenberg-like, then for all $ Z\in \mathcal{Z}$ and every  $X_{m}\in W_{m}(Z), m=1,...,\mu(Z)$,
 \begin{equation}
 [X_{m},j(Z)X_{m}]=(\frac{\theta_{m}(Z)\mid X_{m}\mid}{\mid Z \mid})^{2}Z.
 \end{equation}
 Therefore we have   with the above notation $H_{n}$  under the evolution of the Ricci-Bourguignon flow from Heisenberg type convert to Heisenberg-like type.
%%%%%%%%%%%%%%%%%%%%%%%%%
\subsection{Quaternion Lie groups}
We now recall the construction  of the higher-dimensional, classical quaternion Lie groups. Let $N=Q_{n}$ be a $(4n+3)-$dimensional quaternion group. Let
\begin{eqnarray*}
 x &=& \left( {x_{11} ,x_{21} ,...,x_{n1} ,...,x_{1n} ,x_{2n} ,...,x_{4n} } \right), \nonumber\\
 z &=& \left( {z_1 ,z_2 ,z_3 } \right). \nonumber
 \end{eqnarray*}
Assume that  $q = (x,z) \in N$ and $q' = (x',z') \in N$.  Multiplication on
$N$ is defined as follows:
\begin{eqnarray*}
 L_q (q') &=& L_{(x,z)} (x',z') = (x,z)o(x',z')\\
 &=& \left( x + x',z_1  + z'_1  + \frac{1}{2}(M_1 x,x'),z_2  + z'_2
   + \frac{1}{2}(M_2 x,x')\right.\\&&
\left.,z_3  + z'_3  + \frac{1}{2}(M_3 x,x') \right),
\end{eqnarray*}
where
\begin {equation*}
M_{k}  = \left[ {\begin{array}{*{20}c}
   {A_k } & O &  \cdots  & O  \\
   O & {A_k } &  \ddots  &  \vdots   \\
    \vdots  &  \ddots  &  \ddots  & O  \\
   O & {\; \cdots } & O & {A_k }  \\
\end{array}} \right],
\end{equation*} 
for $k=1,2,3$ and 
\begin{eqnarray*}
A_{1}  = \left[ {\begin{array}{*{20}c}
   0& 1&  0  & 0  \\
   -1 & 0 &  0  &  0   \\
   0 &  0 & 0  & 1  \\
   0 & 0 & -1 & 0  \\
\end{array}} \right],\,\,A_{2}  = \left[ {\begin{array}{*{20}c}
   0& 0 &  0  & -1  \\
   0 & 0 &  -1 & 0   \\
   0  &  1 &  0  & 0  \\
   1 & 0 & 0 & 0  \\
\end{array}} \right],\,\,A_{3}  = \left[ {\begin{array}{*{20}c}
   0&0 & -1  & 0  \\
   0 & 0 &  0  &  1   \\
    1  &  0  &  0  & 0  \\
   0 & -1 & 0 & 0  \\
\end{array}} \right],
\end{eqnarray*}
$ \left( {M_k x,x'} \right)$ is the usual inner product of vectors $ M_{k} x \in {\mathbb{R}}^{4n}$ and $x'\in {\mathbb{R}}^{4n}$.
With respect to this multiplication, we have the following vector fields
\begin{eqnarray}
 X_{1l}  &= &\frac{\partial }{{\partial x_{1l} }} + \frac{1}{2}\left( {x_{2l} \frac{\partial }{{\partial z_1 }}
  - x_{4l} \frac{\partial }{{\partial z_2 }} - x_{3l} \frac{\partial }{{\partial z_3 }}} \right),\nonumber \\
 X_{2l}  &=& \frac{\partial }{{\partial x_{2l} }} + \frac{1}{2}\left( { - x_{1l} \frac{\partial }{{\partial z_1 }}
 - x_{3l} \frac{\partial }{{\partial z_2 }} + x_{4l} \frac{\partial }{{\partial z_3 }}} \right),\nonumber \\
 X_{3l}  &=& \frac{\partial }{{\partial x_{3l} }} + \frac{1}{2}\left( {x_{4l} \frac{\partial }{{\partial z_1 }}
  + x_{2l} \frac{\partial }{{\partial z_2 }} + x_{1l} \frac{\partial }{{\partial z_3 }}} \right), \\
 X_{4l}  &=& \frac{\partial }{{\partial x_{4l} }} + \frac{1}{2}\left( { - x_{3l} \frac{\partial }{{\partial z_1 }}
  + x_{1l} \frac{\partial }{{\partial z_2 }} + x_{2l} \frac{\partial }{{\partial z_3 }}} \right),\nonumber \\
 Z_m  &=& \frac{\partial }{{\partial z_m }},
\end{eqnarray}
for $l = 1,2,...,n$ and $ m = 1,2,3$. 
The nonzero  Lie brackets of vector fields are
\begin{eqnarray}
 & &\left[ X_{1l} ,X_{2l}  \right]=  - Z_{1} ,\qquad   \left[ X_{1l} ,X_{3l}  \right] = Z_{3} 
 ,\qquad\left[X_{1l} ,X_{4l} \right] = Z_{2},  \nonumber \\
  & &\left[ X_{2l} ,X_{3l}  \right]=   Z_{2} ,\qquad   \left[ X_{2l} ,X_{4l}  \right] = -Z_{3}
 ,\qquad\left[X_{3l} ,X_{4l} \right] = -Z_{1}.
 \end{eqnarray}
 Given the above definitions, $Q_{n}$ is
two-step nilpotent. Note the dual of the above vector fields  are as
follows
\begin{equation*}
dx_{kl}, \,\,\, \theta _{r}=dz_{r}-\frac{1}{2}(M_r x,dx),\,\,\,k = 1,2,3,4,\,\,\,1 \le l \le n,\,\,\,r = 1,2,3.
\end{equation*}
Set
$$\mathcal{V}=span\{X_{1l },X_{2l },X_{3l },X_{4l }| 1\leq l\leq n\},\qquad \mathcal{Z}=span\{Z_{1},Z_{2},Z_{3}\}.$$
If we choose an  inner product on ${Q_{n}}$ such that $\mathcal{V}
\bigcup \mathcal{Z}$  is an orthonormal basis for ${Q_{n}}$ then
quaternion Lie group is of Heisenberg type.
%%%%%%%%%%%%%%%%%%%%%%%%%%%%%%%%%%%%%%%%%
\subsection{The Ricci-Burguignon  flow on quaternion Lie group}
Assume that
\begin{eqnarray*}
e_i  &=& X_{1i},\,\,\, e_{n + i}  = X_{2i},\,\,\,e_{2n + i}  = X_{3i},\,\,\,e_{3n + i}  = X_{4i},\\
e_{4n +  r}  &=& Z_{r},\,\,\,i = 1,2,...,n,\,\,\,r = 1,2,3.
\end{eqnarray*}
Let   $g$
be  diagonal and $g_\alpha   = g_{\alpha \alpha }$. 
With the above symbols and $\left[ {e_i ,e_j } \right] = C_{ij}^k e_k$, and Propositions \ref{p1}, \ref{p2} and corallary \ref{c1} we conclued that  the Ricci tensor stay diagonal under the  Ricci-Burguignon  flow, also as follows
\begin{eqnarray}
 R_i  &=&  - \frac{1}{2}\left( \frac{g_{4n + 1}}{ g_{n + i}}  + \frac{g_{4n + 3} }{g_{2n + i}  }+\frac{ g_{4n + 2}}{ g_{3n + i} } \right), \nonumber\\
 R_{n + i}  &=& - \frac{1}{2}\left( \frac{g_{4n + 1}}{ g_{i }} + \frac{g_{4n + 2}}{ g_{2n + i}}  +\frac{ g_{4n + 3}}{ g_{3n + i} } \right), \nonumber\\
 R_{2n + i}  &=&  - \frac{1}{2}\left( \frac{g_{4n + 3}} {g_{i}}  + \frac{g_{4n + 2}}{ g_{n + i}}  +\frac{ g_{4n + 1}}{ g_{3n + i} } \right), \nonumber\\
 R_{3n + i}  &=&  - \frac{1}{2}\left( \frac{g_{4n + 2}}{ g_{i}}  +\frac{ g_{4n + 3}}{ g_{n + i}}  +\frac{ g_{4n + 1}}{ g_{2n + i} } \right), \nonumber\\
 R_{4n + 1} &=& \frac{{g_{4n + 1}^2 }}{2}\sum\limits_{i = 1}^n {\left( \frac{1}{g_i g_{n + i}  }+\frac{1}{ g_{2n + i} g_{3n + i} } \right)}, \nonumber \\
 R_{4n + 2}  &=& \frac{{g_{4n + 2}^2 }}{2}\sum\limits_{i = 1}^n {\left( \frac{1}{g_i g_{3n + i}}  + \frac{1}{g_{n + i} g_{2n + i} } \right)}, \nonumber \\
 R_{4n + 3}  &=& \frac{{g_{4n + 3}^2 }}{2}\sum\limits_{i = 1}^n {\left( \frac{1}{g_i g_{2n + i}}  +\frac{1}{ g_{n + i} g_{3n + i} } \right),} \nonumber
 \end{eqnarray}
and 
\begin{equation*}
R=-\frac{1}{2}\sum'(t),
\end{equation*}
where 
\begin{eqnarray*}
\sum'(t)&=&g_{4n+1}\sum_{i = 1}^{n} \left( \frac{1}{g_i g_{n + i}  }+\frac{1}{ g_{2n + i} g_{3n + i} } \right)
+g_{4n+2}\sum_{i = 1}^{n} \left( \frac{1}{g_i g_{3n + i}  }+\frac{1}{ g_{n + i} g_{2n + i} } \right)\\&&
+g_{4n+3}\sum_{i = 1}^{n} \left( \frac{1}{g_{i} g_{2n + i}  }+\frac{1}{ g_{n + i} g_{3n + i} } \right).
\end{eqnarray*}
Therefore, the Ricci-Bourguignon flow equation, $\frac{{\partial g}}{{\partial
t}} = - 2Ric+2\rho Rg$, on $Q_{n}$ has the  form
\begin{equation}\label{r4}
\begin{cases}
\frac{d}{dt}g_{i}=\frac{g_{4n+1}}{g_{n+i}}+\frac{g_{4n+3}}{g_{2n+i}}+\frac{g_{4n+2}}{g_{3n+i}} -\rho g_{i}\sum',&\text{if $1 \leq i \leq n$;}\\
\frac{d}{dt}g_{n+i}=\frac{g_{4n+1}}{g_{i}}+\frac{g_{4n+2}}{g_{2n+i}}+\frac{g_{4n+3}}{g_{3n+i}}-\rho g_{n+i}\sum',&\text{if $1 \leq i \leq n$;}\\
\frac{d}{dt}g_{2n+i}=\frac{g_{4n+3}}{g_{i}}+\frac{g_{4n+2}}{g_{n+i}}+\frac{g_{4n+1}}{g_{3n+i}}-\rho g_{2n+i}\sum',&\text{if $1 \leq i \leq n$;}\\
\frac{d}{dt}g_{3n+i}=\frac{g_{4n+2}}{g_{i}}+\frac{g_{4n+3}}{g_{n+i}}+\frac{g_{4n+1}}{g_{2n+i}}-\rho g_{3n+i}\sum',&\text{if $1 \leq i \leq n$;}\\
\frac{d}{dt}g_{4n+1}=-\Big(\sum_{i=1}^{n} \frac{g_{4n+1}^{2}}{g_{i}g_{n+i}}+\sum_{i=1}^{n}\frac{g_{4n+1}^{2}}{g_{2n+i}g_{3n+i}}\Big)-\rho g_{4n+1}\sum',\\
\frac{d}{dt}g_{4n+2}=-\Big(\sum_{i=1}^{n} \frac{g_{4n+2}^{2}}{g_{i}g_{3n+i}}+\sum_{i=1}^{n}\frac{g_{4n+2}^{2}}{g_{n+i}g_{2n+i}}\Big)-\rho g_{4n+2}\sum',\\
\frac{d}{dt}g_{4n+3}=-\Big(\sum_{i=1}^{n}
\frac{g_{4n+3}^{2}}{g_{i}g_{2n+i}}+\sum_{i=1}^{n}\frac{g_{4n+3}^{2}}{g_{n+i}g_{3n+i}}\Big)-\rho g_{4n+3}\sum'.
\end{cases}
\end{equation}
%%%%%%%%%%%%%%%%%%%%%%%
\begin{theorem}\label{T2}
Consider the quaternion Lie group $Q_{n}$ with a diagonal  left-invariant metric $g_{0}$. Let $g(t)$ be a solution to the Ricci-Bourguignon flow with initial metric $g_{0}$, then
\begin{description}
  \item[a]  $\frac{d}{dt}\bigg(g_{1}(t) g_{2}(t)...g_{4n}(t)\big(g_{4n+1}(t)g_{4n+2}(t)g_{4n+3}(t)\big)^{\frac{2(1-2n\rho)}{1+3\rho}}\bigg)=0$, 
  \item[b] If $\rho<0$  and $G_{k}(t)=\int_{0}^{t}g_{k}(t)dt$ for $k=4n+1, ..., 4n+3$ then $\mathop{\lim}\limits_{t\to+\infty}G_{k}(t)=+\infty$. 
\item[c] If moreover $g_{j}(0)=g_{1}(0),\,\,\,\, g_{4n+1}(0)=g_{4n+2}(0)=g_{4n+3}(0)$, for $1\leq j\leq 4n$ then a sloution $g(t)$ has the following form
\begin{equation}\label{ss1}
  \begin{cases}
g_{j}(t)=g_{1}(0)\Big( 1+ct  \Big)^{\frac{3(1-2n\rho)}{6+2n-6n\rho}},&\text{if\,\,\,\,$1\leq j\leq 4n$} \\
g_{4n+k}(t)=g_{4n+1}(0)\Big( 1+ct  \Big)^{\frac{-n(1+3\rho)}{3+n-3n\rho}},&\text{if\,\,\,\,$1\leq k\leq 3$}
\end{cases}
\end{equation}
where $c=\frac{g_{4n+1}(0)}{g_{1}^{2}(0)}(6+2n-6n\rho)$.
\end{description}
\end{theorem}
\begin{proof}
\begin{description}
  \item[a]
Assume that  $G(t)=g_{1}(t) g_{2}(t)...g_{4n}(t)\big(g_{4n+1}(t)g_{4n+2}(t)g_{4n+3}(t)\big)^{\frac{2(1-2n\rho)}{1+3\rho}}$, then   we get 
\begin{eqnarray*}
\frac{d}{dt}G(t)&=&\left( \sum_{r = 1}^{4n} \frac{1}{g_{r }}\frac{d g_{r }}{d t} + \frac{2(1-2n\rho)}{1+3\rho}\sum_{k = 1}^{3} \frac{1}{g_{4n + k} }\frac{d g_{4n + k} }{dt} \right)G(t) \\
  &=& \sum_{i = 1}^{n} \left( \frac{1}{g_{i }}\frac{d g_{i }}{d t}
  + \frac{1}{g_{n + i} }\frac{d g_{n + i}}{d t} +   \frac{1}{g_{2n + i} }\frac{d g_{2n + i} }{d t}   + \frac{1}{g_{3n + i} }\frac{d g_{3n + i} }{d t} \right)G(t)\\&&
    + \frac{2(1-2n\rho)}{1+3\rho} G(t)\sum_{k = 1}^{3} \frac{1}{g_{4n + k} }\frac{d g_{4n + k} }{d t}.
\end{eqnarray*}
Using (\ref{r4}), we have 
\begin{eqnarray*}
\frac{d}{dt}G(t)
  &=& \sum_{i=1}^{n}\left(\frac{g_{4n+1}}{g_{i}g_{n+i}}+\frac{g_{4n+3}}{g_{i}g_{2n+i}}+\frac{g_{4n+2}}{g_{i}g_{3n+i}} -\rho \sum'  \right)G(t)\\
&&+\sum_{i=1}^{n}\left(  \frac{g_{4n+1}}{ g_{n+i}g_{i}}+\frac{g_{4n+2}}{ g_{n+i}g_{2n+i}}+\frac{g_{4n+3}}{ g_{n+i}g_{3n+i}}-\rho\sum' \right)G(t)\\
&&+\sum_{i=1}^{n}\left( \frac{g_{4n+3}}{g_{2n+i}g_{i}}+\frac{g_{4n+2}}{g_{2n+i}g_{n+i}}+\frac{g_{4n+1}}{g_{2n+i}g_{3n+i}}-\rho \sum'  \right)G(t)\\
&&+\sum_{i=1}^{n}\left( \frac{g_{4n+2}}{g_{3n+i}g_{i}}+\frac{g_{4n+3}}{g_{3n+i}g_{n+i}}+\frac{g_{4n+1}}{g_{3n+i}g_{2n+i}}-\rho \sum'  \right)G(t)\\
&&+\frac{2(1-2n\rho)}{1+3\rho} \left( -\sum_{i=1}^{n} \frac{g_{4n+1}}{g_{i}g_{n+i}}-\sum_{i=1}^{n}\frac{g_{4n+1}}{g_{2n+i}g_{3n+i}}-\rho \sum' \right)G(t)\\
&&+\frac{2(1-2n\rho)}{1+3\rho} \left(  -\sum_{i=1}^{n} \frac{g_{4n+2}}{g_{i}g_{3n+i}}-\sum_{i=1}^{n}\frac{g_{4n+2}}{g_{n+i}g_{2n+i}}-\rho \sum' \right)G(t)\\
&&+\frac{2(1-2n\rho)}{1+3\rho} \left(  -\sum_{i=1}^{n}\frac{g_{4n+3}}{g_{i}g_{2n+i}}-\sum_{i=1}^{n}\frac{g_{4n+3}}{g_{n+i}g_{3n+i}}\Big)-\rho \sum' \right)G(t)=0.
\end{eqnarray*}
\item[b] For $\rho<0$ the first four equations (\ref{r4}) yield  that $g_{j},\,\,\, 1\leq j \leq 4n$ is an increasing function, so
  \begin{eqnarray*}
\sum\nolimits_{1}&=&\sum_{i = 1}^{n} \left( \frac{1}{g_i g_{n + i}  }+\frac{1}{ g_{2n + i} g_{3n + i} } \right),\,\,\,
\sum\nolimits_{2}=\sum_{i = 1}^{n} \left( \frac{1}{g_i g_{3n + i}  }+\frac{1}{ g_{n + i} g_{2n + i} } \right),\\
\sum\nolimits_{3}&=&\sum_{i = 1}^{n} \left( \frac{1}{g_{i} g_{2n + i}  }+\frac{1}{ g_{n + i} g_{3n + i} } \right)
\end{eqnarray*}
are  positive and decreasing functions of $t$.    Since  $g_{4n+k}(t),\,\,\,\, 1\leq k\leq 3$ are  positive we have
  \begin{eqnarray*}
\frac{d}{dt}g_{4n+1}(t)&=&-g_{4n+1}^{2}\sum_{i=1}^{n}\Big( \frac{1}{g_{i}g_{n+i}}+\frac{1}{g_{2n+i}g_{3n+i}}\Big)-\rho g_{4n+1}\sum'\\
&\geq&-g_{4n+1}^{2}\sum_{i=1}^{n}\Big( \frac{1}{g_{i}g_{n+i}}+\frac{1}{g_{2n+i}g_{3n+i}}\Big)\geq -g_{4n+1}^{2}\sum\nolimits_{1}(0)
\end{eqnarray*}
which by direct computation implies that  
\begin{equation*}
 g_{4n+1}(t)\geq \frac{1}{\sum\nolimits_{1}(0)t+g^{-1}_{4n+1}(0)},
\end{equation*}
therefore
\begin{equation*}
\mathop{\lim}\limits_{t\to+\infty}G_{4n+1}(t)=\mathop{\lim}\limits_{t\to+\infty}\int_{0}^{t}g_{4n+1}(r)dr\geq \mathop{\lim}\limits_{t\to+\infty}\int_{0}^{t}\frac{1}{\sum\nolimits_{1}(0) r+g^{-1}_{4n+1}(0)}dr= +\infty,
\end{equation*}
similarly $\mathop{\lim}\limits_{t\to+\infty}G_{4n+2}(t)=+\infty$ and $\mathop{\lim}\limits_{t\to+\infty}G_{4n+3}(t)=+\infty$.
 
\item[c] $g_{j}(t)$ and $g_{4n+k}(t)$ for $1\leq j\leq 4n,\,\,\,\, 1\leq k\leq 3$  given in (\ref{ss1}) satisfy in  the Ricci-Bourguignon flow (\ref{r4}).

\end{description}
\end{proof}
%%%%%%%%%%%%%%%%%%
\begin{proposition}\label{p6}
Heisenberg type of  Lie group ${Q}_{n}$ is not preserved under
the evolution of the Ricci-Bourguignon flow with solution (\ref{ss1}) and additional condition  $g_{1}^{2}(0)=g_{4n+1}(0)$.
\end{proposition}
\begin{proof} The proof  is similar to proof of Proposition  \ref{p4}.
 In $(Q_{n}, g_{t}=< , >_{t})$, we have
\begin{eqnarray*}
j(Z_{1})e_{i}=-\frac{g_{4n+1}(t)}{g_{n+i}(t)}e_{n+i},\,\,\, j(Z_{1})e_{n+i}=\frac{g_{4n+1}(t)}{g_{i}(t)}e_{i},\,\,\, j(Z_{1})e_{2n+i}=-\frac{g_{4n+1}(t)}{g_{3n+i}(t)}e_{3n+i},\\
j(Z_{1})e_{3n+i}=\frac{g_{4n+1}(t)}{g_{2n+i}(t)}e_{2n+i},\,\,\, j(Z_{2})e_{i}=\frac{g_{4n+2}(t)}{g_{3n+i}(t)}e_{3n+i},\,\,\, j(Z_{2})e_{n+i}=\frac{g_{4n+2}(t)}{g_{2n+i}(t)}e_{2n+i},\\
j(Z_{2})e_{2n+i}=-\frac{g_{4n+2}(t)}{g_{n+i}(t)}e_{n+i},\,\,\, j(Z_{2})e_{3n+i}=-\frac{g_{4n+2}(t)}{g_{i}(t)}e_{i},\,\,\, j(Z_{3})e_{i}=\frac{g_{4n+3}(t)}{g_{2n+i}(t)}e_{2n+i},\\
j(Z_{3})e_{n+i}=-\frac{g_{4n+3}(t)}{g_{3n+i}(t)}e_{3n+i},\,\,\, j(Z_{3})e_{2n+i}=-\frac{g_{4n+3}(t)}{g_{i}(t)}e_{i},\,\,\, j(Z_{3})e_{3n+i}=-\frac{g_{4n+3}(t)}{g_{n+i}(t)}e_{n+i}.
\end{eqnarray*}
By Theorem \ref{T2}  for $1\leq i\leq 4n$ and $1\leq k \leq3$ we have  $g_{i}(t)=g_{1}(t)$ and  $g_{4n+k}(t)=g_{4n+1}(t)$. Therefore, if we set $E =\frac{g_{4n+1}(t)}{g_{i}(t)}$, then
\[
\begin{array}{l}
 j(Z_1 ) = E\left[ {\begin{array}{*{20}c}
   o & { I_n } & o & o  \\
   {-I_n } & o & o & o  \\
   o & o & o & {  I_n }  \\
   o & o & {-I_n } & o  \\
\end{array}} \right],\begin{array}{*{20}c}
    & {} & {j(Z_2 ) = E\left[ {\begin{array}{*{20}c}
   o & o & o & {  -I_n }  \\
   o & o & {-I_n } &o   \\
  o  & {I_n } & o & o  \\
 {I_n }   &o  & o & o  \\
\end{array}} \right],}  \\
\end{array} \\
 j(Z_3) = E\left[ {\begin{array}{*{20}c}
   o & o & {- I_n }  & o \\
   o & o  & o& { I_n }  \\
  { I_n } & o &  o & o  \\
    o &{-I_n } & o & o  \\
\end{array}} \right], \\
 \end{array}
\]
hence for any real constants $c_{1}, c_{2}$ and $c_{3}$, we find that
\[
j(c_1 Z_1  + c_2 Z_2  + c_3 Z_3 ) = E\left[ {\begin{array}{*{20}c}
   o & { c_1 I_n } & { - c_3 I_n } & { - c_2 I_n }  \\
   {- c_1 I_n } & o & { - c_2 I_n } & {c_3 I_n }  \\
   {c_3 I_n } & {c_2 I_n } & o & { c_1 I_n }  \\
   {c_2 I_n } & { - c_3 I_n } & {- c_1 I_n } & o  \\
\end{array}} \right].
\]
If $Z = c_1 Z_1  + c_2 Z_2  + c_3 Z_3$ then
\begin{equation*}
 (j(Z))^{2}=  - E^{2} (c_{1}^{2}  + c_{2}^{2}  + c_{3}^{2} )I_{4n}  =  - E^{2} \frac{| Z |_{t}^{2} }{g_{4n+1}(t)}I_{4n}
 = -\frac{1}{(6+2n-6n\rho)t+1}| Z |_{t}^{2}I_{4n}.
 \end{equation*}
Hence, it is not of Heisenberg type.
\end{proof}
%%%%%%%%%%%%%%
\begin{remark}
By the Definition  \ref{ddh}, $Q_{n}$ along the Ricci-Bourguignon flow from Heisenberg type convert to Heisenberg-like type 
\end{remark}
%%%%%%%%%%%%%%%
 \section{Deformation  of marked length spectrum    }
Suppose that the  Lie group $N$ is $H_{n}$ or $Q_{n}$ and $g(t)$ is the  solution of the Ricci-Bourguignon flow  in (\ref{ss}) and (\ref{ss1}) respectively, with some conditions given in Proposition \ref{p4} and  \ref{p6}.
Then for $t=0$ we have
\begin{equation*}
    j(Z)^{2}=-|Z|^{2}Id \,\,\,\,\,  \text{for all $Z\in \mathcal{Z}$},
\end{equation*}
that is the group $N$ is  Heisenberg type.
But if in Heisenberg Lie group $(H_{n}, g(t))$ suppose that $    \eta_{t}=\frac{1}{(n+2-n\rho)t+1}$,
then in $(H_{n}, g(t))$ from the proof of Proposition \ref{p4}  we have
\begin{equation*}
    j(Z)^{2}=-\eta_{t}|Z|_{t}^{2}Id\,\,\,\,\,  \text{for all $Z\in \mathcal{Z}$}.
\end{equation*}
Also if in the  quaternion  Lie group $(Q_{n}, g(t))$ we suppose that $\zeta_{t}= \frac{1}{(6+2n-6n\rho)t+1}$
then in $(Q_{n}, g(t))$  from the proof of Proposition \ref{p6}  we obtain
\begin{equation*}
    j(Z)^{2}=-\zeta_{t}|Z|_{t}^{2}Id\,\,\,\,\,  \text{for all $Z\in \mathcal{Z}$}.
\end{equation*}
Let $P_{t}=\eta_{t}$ or $\zeta_{t}$. For ${H}_{n}$  or ${Q}_{n}$ we have $j(Z)^{2}=-P_{t}|Z|_{t}^{2}Id$.
Similarly  the argument of \cite{SA}, we conclude the following statments about the deformation of spectrum and length spectrum.  The spectrum and the length spectrum have relationship with each other (see \cite{GM, GM1, GM2}). By easy computation we have 
%%%%%%%%%%%%%%%%%%%%%%%%%%%%%%%%%%%%%%%%%%%%%%%%%%%%%%%%%%%%%%%%%%%%%%%%%%%%%%%%%%%%%%%%%%%%%%%%%%%%%%%%%%%%%%%%%
\begin{proposition}\label{p8}
Let $({\mathcal{N}, < , >_{t})}$ is the Lie algebra of $N$ where $N$ is $H_{n}$ or $Q_{n}$ . Then  we have
\begin{enumerate}
  \item $<j(Z)X,j(Z^{*})X>_{t}=P_{t}<Z,Z^{*}>_{t}<X,X>_{t}$ for all $Z,Z^{*}\in \mathcal{Z}$ and $X\in \mathcal{V}$;
  \item $<j(Z)X,j(Z)Y>_{t}=P_{t}<Z,Z>_{t}<X,Y>_{t}$ for all $Z\in \mathcal{Z}$ and $X,Y\in \mathcal{V}$;
  \item $|j(Z)X|_{t}=P_{t}^{\frac{1}{2}}|Z|_{t}|X|_{t}$ for all $Z\in \mathcal{Z}$ and $X\in \mathcal{V}$;
  \item $j(Z)\circ j(Z^{*})+j(Z^{*})\circ j(Z)=-2P_{t}<Z,Z^{*}>_{t}Id$ for all $Z,Z^{*}\in \mathcal{Z}$;
  \item $[X,j(Z)X]=P_{t}<X,X>_{t}Z$ for all $Z\in \mathcal{Z}$ and $X\in \mathcal{V}$.
\end{enumerate}
\end{proposition}
%%%%%%%%%%%%%%%%%%%%
\begin{proposition}\label{p9}
Let   $\sigma(s,t)=exp(X(s,t)+Z(s,t))$  
 be a curve in $2$-step nilpotent Lie group with left invariant metric  $(N, g(t))$  where $N$ is $H_{n}$ or $Q_{n}$, such that $\sigma(0,t)=e$ and $\sigma'(0,t)=X_{0}(t)+Z_{0}(t)$, where $X_{0}(t)\in \mathcal{V}(t)$, $Z_{0}(t)\in \mathcal{Z}(t)$ and $e$ is the identity in $N$. Let $g(t)$ is the  solution of the Ricci-Bourguignon flow  on $H_{n}$ and  $Q_{n}$  in (\ref{ss}) and (\ref{ss1}) respectively. Then
\begin{equation}\label{R13}
\begin{cases}
 X(s,t)=(coss\theta-1)J^{-1}X_{0}(t)+\frac{sins\theta}{\theta}X_{0}(t)\\
 Z(s,t)=\Bigg(s(1+\frac{|X_{0}(t)|_{t}^{2}}{2|Z_{0}(t)|_{t}^{2}})+\frac{sins\theta}{\theta}\frac{|X_{0}(t)|_{t}^{2}}{2| Z_{0}(t)|_{t}^{2}}  \Bigg)Z_{0}(t)
 \end{cases}
\end{equation}
where $J=j(Z_{0}(t))$, $\theta=\sqrt{P_{t}}|Z_{0}(t)|_{t}$.
\end{proposition}
%%%%%%%%%%%%%%%%%%%%%%%%%%%%%%%%%%%%%%%%%%%%%%%%%%%%%%%%%%%%%%%%%%%%%%%%%%%%%%%%%%%%%%%%%%%%%%%%%%%%%%%%%%%%%%%%%%%%%%%%%%%%%%%%%%%%%%%%%%%
\begin{definition}
A nonidentity elements $\varphi(t)$ of $(N, g(t))$ translates a unit speed geodesic $\sigma(s,t)$ in $(N, g(t))$ by an amount $\omega(t)>0$ if $\varphi(t).\sigma(s,t)=\sigma(s+\omega(t),t)$ for all $s\in \mathbb{R}$. The amount $\omega(t)$ is called a period of $\varphi(t)$.
\end{definition}
%%%%%%%%%%%%%
\begin{definition}
Let $N$ be a simply connected, nilpotent Lie group with a left invariant metric, and let $\Gamma\subseteq N $ be a discrete subgroup of $N$. The group $\Gamma$ is said to be
a lattice in $N$ if the quotient manifold    $\Gamma\setminus N$ obtained by letting $\Gamma$ act on $N$ by left translation is compact.
\end{definition}
%%%%%%%%%%%%%%%%%%%%%%%%%%%%%%%%%%%%%%%%%%%%%%%%%%%%%%%%%%%%%%%%%%%%%%%%%%%%%%%%%%%%%%%%%%%%%%%%%%
\begin{proposition}\label{pr}
Let $(N,g(t))$ be $(H_{n},g(t))$ or $(Q_{n},g(t))$, $g(t)$ is the  solution of the Ricci-Bourguignon flow    in (\ref{ss}) and (\ref{ss1}) respectively,  and $\Gamma$  be a discrete subgroup of $N$.
Let  $\varphi(t)\in \Gamma $ be a family of nonidentity elements of the center of $N$, such that $log \varphi(t)\in\mathcal{Z}$. Then $\varphi(t)=exp(V^{*}(t)+Z^{*}(t))$
has the following periods.
\begin{equation*}
  \left\{|Z^{*}(t)|_{t}, \sqrt{(4\pi k)(|Z^{*}(t)|_{t}-\pi k)}; \text{where $k $ is an integer and } \{1\leq k\leq\frac{1}{2\pi}|Z^{*}(t)|_{t}\}\right\}
\end{equation*}
\end{proposition}
\begin{proof}
Every unit speed geodesic of $N$ is translated by some element $\varphi(t)$ of $N$ (see~\cite{E1} ) and  (\ref{R13}) proves the proposition.
\end{proof}
%%%%%%%%%%%%%%%%%%%%%%%%%%%%%%%%%%%%%%%%%%%%%%%%%%%%%%%%%%%%%%%%%%%%%%%%%%%%%%%%%%%%%%%%%%%%%%%%%%

%%%%%%%%%%%%%%%%%%%%%%%%%%%%%%%%%%%%%%%%%%%%%%%%%%%%%%%%%%%%%%%%%%%%%%%%%%%%%%%%%%%%%%%%%%%%%%%%%%%%%%%%%%%%%%%%%%%%%%%%%%%%%%%%%%%%%%%%%%%%%%%%%%%%%%%%%%%%%%%%%%%%%%%%%%%%%%%%%%%%
\begin{definition}
Let $M$ be a compact Riemannian manifold. For each nontrivial free homotopy class $C$  of closed curves in $M$ we define $l(C)$ to be the collection of all lengths of smoothly
closed geodesics that belong to $C$.
\end{definition}
%%%%%%%%%%%%%%%%%%%%%%%%%%%%%%%%%%%%%%%%%%%%%%%%%%%%%%%%%%%%%%%%%%%%%%%%%%%%%%%%%%%%%
\begin{definition}
The length spectrum of a compact Riemannian manifold $M$ is  the collection of all ordered pairs $(L,m)$, where $L$ is
the length of a closed geodesic in $M$ and $m$ is the multiplicity of $L$, i.e. $m$ is the number of free homotopy classes $C$ of closed curves in $M$ that contain a closed geodesic of length  $L$.
\end{definition}
%%%%%%%%%%%%%%%%%%%%%%%%%%%%%%%%%%%%%%%%%%%%%%%%%%%%%%%%%%%%%%%%%%%%%%%%%%%%%%%%%%%%%%%%%%%%%%%%%%%%%%
\begin{lemma}Let  $g(t)$ be find the solution of the  Ricci-Bourguignon flow in (\ref{ss}) and (\ref{ss1}). Then $(\Gamma\setminus H_{n},g(t))$
and $(\Gamma\setminus H_{n},g_{0})$ have the same length spectrum, also $(\Gamma\setminus Q_{n},g(t))$
and $(\Gamma\setminus Q_{n},g_{0})$ have the same length spectrum.
\end{lemma}
\begin{proof} Let $(N,g(t))$ be $(H_{n},g(t))$ or $(Q_{n},g(t))$. If $\varphi(t)$ belongs to a discrete group $\Gamma\subseteq N $, then the periods of  $\varphi(t)$ are precisely the lengths of the closed geodesic in $\Gamma\setminus N$
that belong to the free homotopy class of closed curves in $\Gamma\setminus N$ determined by $\varphi(t)$. Therefore a free homotopy class of closed curves in $\Gamma\setminus N$ corresponds to a conjugate class of an element $\varphi$
in  $\Gamma$ and the collection $l(C)$ is then  precisely the set of periods of $\varphi$. For any nonidentity elements  $\varphi(t)=exp(V^{*}(t)+Z^{*}(t))\in N$ that does not lie in the center of $N$, by Lemma 3.2 in \cite{SA}  it has a unique  period $\omega(t)=|V^{*}(t)|_{t}$.
Therefore in Heisenberg Lie group $(H_{n},g(t)))$, if we suppose  that $V^{*}(t)=\Sigma_{i=1}^{n}a_{i}e_{i}+b_{i}e_{n+i}$ for some $a_{i},b_{i}\in \mathbb{R}$, then we obtain
\begin{equation*}
  |V^{*}(t)|_{t}^{2} =g_{1}(t)\sum_{i=1}^{n}(a_{i}^{2}+b_{i}^{2})=(1+bt)^{\frac{1-n\rho}{n+2-n\rho}}|V^{*}(t)|_{0}^{2},
\end{equation*}
where $b=(n+2-n\rho)\frac{g_{N}(0)}{g_{1}(0)g_{n+1}(0)} $ and in quaternion Lie group we suppose that
\begin{equation*}
  V^{*}(t)=\Sigma_{i=1}^{n}a_{i}X_{1i}+b_{i}X_{2i}+c_{i}X_{3i}+d_{i}X_{4i}\,\,\,\, \text{for some $a_{i},b_{i},c_{i},d_{i}\in \mathbb{R}$,}
\end{equation*}
then
\begin{eqnarray*}
   |V^{*}(t)|_{t}^{2} &=&\Sigma_{i=1}^{n}a_{i}^{2}|X_{1i}|_{t}^{2}+b_{i}^{2}|X_{2i}|_{t}^{2}+c_{i}^{2}|X_{3i}|_{t}^{2}+d_{i}^{2}|X_{4i}|_{t}^{2}  \\
   &=&  (1+ct)^{\frac{3(1-2n\rho)}{6+2n-6n\rho}} |V^{*}(t)|_{0}^{2}.
\end{eqnarray*}
where $c=(6+2n-6n\rho)\frac{g_{4n+1}(0)}{g_{1}^{2}(0)}$.
Let
\begin{equation*}
  W^{*}(t)=(1+ct)^{-\frac{3(1-2n\rho)}{12+4n-12n\rho}} V^{*}(t)
\end{equation*}
and $\psi(t)=exp(W^{*}(t)+Z^{*}(t))$ then $|W^{*}(t)|_{t}=|V^{*}(t)|_{0}$ in $(Q_{n},g(t))$. Similarly, if 
$W^{*}(t)=(1+bt)^{-\frac{1-n\rho}{2n+4-2n\rho}} V^{*}(t)$ then in $(H_{n},g(t))$ we have $|W^{*}(t)|_{t}=|V^{*}(t)|_{0}$. 
 Hence the period of $\psi(t)$ is $\omega(t)$.
 Also, since for arbitrary nonidentity elements  $\varphi(t)=exp(V^{*}(t)+Z^{*}(t))\in N$ which are  in the center of $N$,
  we have the following periods.
\begin{equation*}
  \left\{|Z^{*}(t)|_{t}, \sqrt{(4\pi k)(|Z^{*}(t)|_{t}-\pi k)}; \text{where $k $ is an integer and } \{1\leq k\leq\frac{1}{2\pi}|Z^{*}(t)|_{t}\}\right\}.
\end{equation*}
Therefore in Heisenberg Lie group $(H_{n},g(t)))$ we see that $Z^{*}(t)=ae_{2n+1}$ for some $a\in \mathbb{R}$. We obtain
\begin{eqnarray*}
  |Z^{*}(t)|_{t}^{2} &=&  a^{2}|e_{2n+1}|_{t}^{2}=(1+bt)^{-\frac{n+n\rho}{n+2-n\rho}}|Z^{*}(t)|_{0}^{2}
\end{eqnarray*}
and in quaternion Lie group we suppose that
\begin{equation*}
  Z^{*}(t)=\Sigma_{i=1}^{3}a_{i}Z_{4n+i}\,\,\,\, \text{for some $a_{i}\in \mathbb{R}$,}
\end{equation*}
then
\begin{equation*}
   |Z^{*}(t)|_{t}^{2} =\Sigma_{i=1}^{3}a_{i}^{2}|Z_{4n+i}|_{t}^{2}=  (1+ct)^{\frac{-n(1+3\rho)}{n+3-3n\rho}} |Z^{*}(t)|_{0}^{2}.
\end{equation*}
Then in any case the set of periods of $\varphi(t)$ is similar and this implies that length spectrum on $(H_{n},g_{0})$ or $(Q_{n},g_{0})$ is preserved under the  metric in (\ref{ss}) and (\ref{ss1}).
\end{proof}
%%%%%%%%%%%%%%%%%%%%%%%%%%%%%%%%%%%%%%%%%%%%%%%%%%%%%%%%%%%%%%%%%%%%%%%%%%%%%%%%%%%%%%%%%%%%%%%%%%%%%%%
\begin{definition} Two Riemannian manifolds $M_1$ and $M_2$ are said to have the
same marked length spectrum if there exists an isomorphism (called a marking)
$T:\pi_{1}(M_1)\rightarrow \pi_{1}(M_2)$ such that, for each $\gamma\in\pi_{1}(M_1)$, the collection of lengths(counting multiplicities) of closed geodesics in the free homotopy class $[\gamma]$ of $M_1$
coincides with the analogous collection in the free homotopy class $[T(\gamma)]$ of $M_2$, i.e. $l(T_{*}(C))=l(C)$ for all nontrivial free homotopy classes of closed curves in $M_1$, where $T_{*}$ denotes the induced map on free homotopy classes.
\end{definition}
%%%%%%%%%%%%%%%%%%%%%%%%%%%%%%%%%%%%%%%%%%%%%%%%%%%%%%%%%%%%%%%%%%%%%%%%%%%%%%%%%%%%%%%%%%%%%%%%%%%%%%%%%%%%%%%%%%%%%%%%%%%%%%%%%%%%%%%%%%%%%%%%%%%%%%%%%%
\begin{definition}Two Riemannian manifolds $ (M_{1}, g_{1})$ and $(M_{2}, g_{2})$ are said to have $C^k$-conjugate geodesic
flows if there is a $C^k$ diffeomorphism $F :S(M_{1}, g_{1})\rightarrow S(M_{2}, g_{2})$ between their
unit tangent bundles that intertwines their geodesic flows i.e.,
$F\circ G_{M_{1}}^{s}=G_{M_{2}}^{s}\circ F$
where $ G_{M_{1}}^{s}$ and $G_{M_{2}}^{s}$ are geodesic flow of $M_{1}$ and $M_{2}$ respectively.
\end{definition}
\begin{definition}
 A compact Riemannian
manifold $M$ is said to be $C^k$-geodesically rigid within a given class $\mathcal{M}$ of Riemannian manifolds
if any Riemannian manifold $M_{1}$ in $\mathcal{M}$ whose geodesic flow is $C^k$-conjugate to that
of $M$ is isometric to $M$.
\end{definition}
\begin{definition}
The solution $g(t)$ of  the Ricci-Bourguignon flow with the  initial condition $g(0)=g_{0}$ is called a Ricci-Bourguignon
soliton if there exist a smooth function $u(t)$ and a $1$-parameter family of diffeomorphisms  ${\psi_{t}}$ of $M^{n}$
such that
\begin{equation*}
g(t)=u(t)\psi^{*}_{t}(g_{0}),\,\,u(0)=1,\,\, \psi_{0}=id_{M^{n}}.
\end{equation*}
\end{definition}
%%%%%%%%%%%%%%%%%
By similar to the  proof of  Theorems 3.1 and 3.2 in \cite{SA}, we have the following lemma.
%%%%%%%%%%%%%%%%%%%%%%%%%%%%%%%%%%%%%%%%%%%%%%%%%%%%%%%%%%%%%%%%%%%%%%%%%%%%%%%%%%%%%%%%%%
\begin{lemma}
{\bf a)}The spectrum and marked length spectrum on a compact nilmanifold is preserved under the Ricci-Bourguignon soliton.\\
{\bf b)}The geodesically rigidity on compact nilmanifold of Heisenberg type is invariant under the  Ricci-Bourguignon soliton.
\end{lemma}

%%%%%%%%%%%%%%%%%%%%%%%%%%%%%%%%%%%%%%%%%%%%%%%%%%%%%%%%%%%%%%%%%%%%%%%%%%%%%%%%%%%%%%%%%
%%%%%%%%%%%%%%%%%%%%%%%%%


\begin{thebibliography}{99}
\bibitem{SA}S.~Azami,   A.~Razavi, \it{ Deformation of  spectrum and length spectrum on some compact nilmanifolds  under the Ricci flow}, Eurasian Math. J., 9(1)(2018),11-29.
\bibitem{B} A. ~Besse, \it{ Einstein manifolds}, Ergeb. Math. {\bf10} (1987), Springer-Verlag, Berlin-Heidelberg.
\bibitem{JPB} J. P. Bourguignon,\it{ Ricci curvature and Einstein metrics}, Global differential geometry and global analysis (Berlin,1979) Lecture nots in Math. vol. 838, Springer, Berlin, (1981), 42-63.
\bibitem{GC} G. Catino, L. Cremaschi, Z. Djadli, C. Mantegazza and  L. Mazzieri, \it{The Ricci-Bourguignon flow}, Pacific J. Math. (2015).
\bibitem{E1}  P. ~Eberlein, \it{ Geometry of 2-step nilpotent groups with a left invariant metric}, Ann. Sci.  Ecole Norm. Sup. (4) {\bf27} (1994), 611-660.
\bibitem{E2} P.~Eberlein, \it{ Geometry of 2-step nilpotent groups with a left invariant metric. II}, Trans. Amer. Soc. Vol. {\bf343},  no.2, (1994), 805-828.
\bibitem{GM} R.~Gornet, M.~B.~Mast, \it{The length spectrum of Riemannian two-step nilmanifolds}, Ann. Sci. de l'\'{E}.N.S. {\bf4} (2000), no.2, 181-209.

\bibitem{GM1} R.~Gornet, Y.~Mao, \it{Geodesic conjugacies of two-step nilmanifolds}, Mich. Math. J. {\bf 45} (1998), 451-481.

\bibitem{GM2} C.~Gordon, Y.~Mao, \it{Comparisons of Laplace spectra, length spectra and geodesic flows of some Riemannian manifolds}, Mathematics Research Letters 1, (1994),  677-688.

\bibitem{He} S.~Helgason,   \it {Differential geometry, Lie groups, and symmetric spaces}, Academic press, New York, 2001.
\bibitem{JLL}J. Lauret, \it{Ricci soliton homegenous nilmanifolds}, Math. Ann. {\bf 319} (2001), no.4, 715-733.
\bibitem{M} J.~Milnor,  \it {Curvature of Left-invariant Metrics on Lie Groups}, Adv. in Math. {\bf21} (1976), 293-329.
\bibitem{PPY} T.L. Payen, \it{The existence of  soliton metrics for nilpotent Lie groups}, Geom. Dedic. {\bf145} (2010), 71-88.
\bibitem{PPY1} T.L. Payen, \it{The Ricci flow for nilmanifolds}, J. Model. Dyn. {\bf 4} (2010), no.1, 65-90.
\bibitem{W} M. B. ~Williams,  \it{ Explicit Ricci solitons on nilpotent Lie group}, Journal of Geometric Analysis, ( 2011), 1-26.
%%%%%%%%%%%%%%%%%%%%%%%
%%%%%%%%%%%%%%%%%%%%%%%%%%


\end{thebibliography}
\end{document}